\documentclass[11pt, a4paper]{article}
\usepackage[total={16.2cm,25cm}]{geometry}
\usepackage{amsmath}
\usepackage{amssymb}
\usepackage{enumerate}
\usepackage{amsthm}
\usepackage{graphicx}

\usepackage[utf8]{inputenc}
\usepackage{color}

\makeatletter
\newcommand{\makecircled}[2][\mathord]{#1{\mathpalette\make@circled{#2}}}
\newcommand{\make@circled}[2]{%
  \begingroup\m@th
  \sbox\z@{$#1A$}%
  \sbox\tw@{%
    \raisebox{\depth}{%
      \vphantom{$#1A$}%
      \ooalign{%
        \hidewidth$#1\make@smaller{#1}{#2}$\hidewidth\cr
        $#1\bigcirc$\cr
      }%
    }%
  }%
  \resizebox{!}{\ht\z@}{\box\tw@}%
  \endgroup
}
\newcommand{\make@smaller}[2]{%
  \vcenter{\hbox{\scalebox{0.7}{$\m@th#1#2$}}}%
}
\makeatother

\newcommand{\Crr}{\mathbb{C}^{r\times r}} 
\newcommand{\Cnn}{\mathbb{C}^{n\times n}} 
\newcommand{\Crnr}{\mathbb{C}^{r\times (n-r)}} 
\newcommand{\Crjrj}{\mathbb{C}^{r_j \times r_j}}
\newcommand{\Crkjrkj}{\mathbb{C}^{r_{kj} \times r_{kj}}}
\newcommand{\Cmn}{\mathbb{C}^{m \times n}}
\newcommand{\Cnm}{\mathbb{C}^{n \times m}}

\newcommand{\rk}{\operatorname{rk}}

\theoremstyle{plain}
\newtheorem{theorem}{Theorem}
\newtheorem{corollary}[theorem]{Corollary}
\newtheorem{proposition}[theorem]{Proposition}
\newtheorem{lemma}[theorem]{Lemma}

\theoremstyle{definition}
\newtheorem{definition}[theorem]{Definition}
\newtheorem{example}[theorem]{Example}
\newtheorem{remark}[theorem]{Remark}

\begin{document}

\title{Lattice properties of the sharp partial order}

\author{Cecilia R. Cimadamore\footnote{Departamento de Matemática, Universidad Nacional del Sur (UNS), Bahía Blanca, Argentina. Instituto de Matemática (INMABB), Universidad Nacional del Sur (UNS)-CONICET, Bahía Blanca, Argentina. https://orcid.org/0009-0007-5252-2352, E-mail: {\tt crcima@criba.edu.ar}.}, 
 \;  
 Laura A. Rueda\footnote{Departamento de Matemática, Universidad Nacional del Sur (UNS), Bahía Blanca, Argentina. Instituto de Matemática (INMABB), Universidad Nacional del Sur (UNS)-CONICET, Bahía Blanca, Argentina. E-mail: {\tt laura.rueda@uns.edu.ar}.}, 
 \;  N\'estor Thome\footnote{Instituto Universitario de Matem\'atica Multidisciplinar, Universitat Polit\`ecnica de Val\`encia, 46022,
  Valencia, Spain. E-mail: {\tt njthome@mat.upv.es}.}, 
  \;
   Melina V. Verdecchia\footnote{Departamento de Matemática, Universidad Nacional del Sur (UNS), Bahía Blanca, Argentina. Instituto de Matemática (INMABB), Universidad Nacional del Sur (UNS)-CONICET, Bahía Blanca, Argentina. E-mail: {\tt mverdec@uns.edu.ar}.}.
}

\date{}

\maketitle

\begin{abstract}The aim of this paper is to study lattice properties of the sharp partial order for complex matrices having index at most 1. We investigate the down-set of a fixed matrix $B$ under this partial order via isomorphisms with two different partially ordered sets of projectors. These are, respectively, the set of projectors that commute with a certain (nonsingular) block of a Hartwig-Spindelböck decomposition of $B$ and the set of projectors that commute with the Jordan canonical form of that block. Using these isomorphisms, we study the lattice structure of the down-sets and we give properties of them. Necessary and sufficient conditions under which the down-set of B is a lattice were found, in which case we describe its elements completely. We also show that every down-set of $B$ has a distinguished Boolean subalgebra and we give a description of its elements. We characterize the matrices that are above a given matrix in terms of its Jordan canonical form. Mitra (1987) showed that the set of all $n \times n$ complex matrices having index at most 1 with $n\geq 4$ is not a lower semilattice. We extend this result to $n=3$ and prove that it is a lower semilattice with $n=2$. We also answer negatively a conjecture given by Mitra, Bhimasankaram and Malik (2010). As a last application, we characterize solutions of some matrix equations via the established isomorphisms.
\end{abstract}

AMS Classifications: 06A06 (primary), 15A09 (secondary), 15A21 (secondary).  

Keywords: Sharp partial order, Hartwig-Spindelböck factorization, lattice structure, Jordan canonical form.

\section{Introduction}\label{sec:introduccion2}

In this paper we work with the sharp partial order defined on the set of complex matrices that have group inverse. This order were defined by Mitra in \cite{Mi}. Later, Mitra, Bhimasankaram and Malik made in \cite{MiBhMa} an extensive study of this partial order and obtained interesting properties. The sharp partial order was also studied by other authors. Gro{\ss} in \cite{Grob} worked with this partial order and characterized the matrices that are above a given matrix using the Hartwig-Spindelböck decomposition of the matrices. More recently, Marovt in \cite{Mar} and Rakić and Djordjević \cite{RaDj} studied this partial order in a ring with involution. Cvetkovi\'{c}-Ili\'{c}, Mosi\'{c} and Wei \cite{CvMoWe}, and, Jose and Sivakumar \cite{ShSi} analized the sharp order for the case of bounded linear Hilbert space operators.

The set of complex $ m \times n $ matrices is denoted by $ \mathbb{C}^{m \times n} $. The conjugate transpose, range, and rank of $ A \in \Cmn $ are denoted by $A^*$, $ \mathcal{R}(A) $, and $ \rk(A) $, respectively. The set of all $n \times n$ complex matrices having index at most 1, that is, $\rk(A^2)= \rk(A) $, is denoted by $ \mathbb{C}_1^n$. It is known that $A\in \mathbb{C}_1^n$ if and only if $A$ has group inverse, that is, there exists a (unique) square matrix $A^\sharp$ that satisfies $AA^\sharp A=A$, $A^\sharp A A^\sharp=A^\sharp$, and $AA^\sharp=A^\sharp A$. The identity matrix of order $n\times n$ is denoted by $I_n $ and zero matrices are denoted simply by $O$. Let us recall that given a matrix $B$ where $0< r = \rk (B)$ and the $ r $ positive singular values $ \sigma_1, \dots , \sigma_r $ of $ B $ are ordered in decreasing order, a Hartwig-Spindelböck decomposition of $ B $ (see \cite{HaSp}) is given by
\begin{equation} \label{hartwigB}
B= U \begin{bmatrix}
\Sigma K & \Sigma L \\
O & O \\
\end{bmatrix} U^*,
\end{equation}
where $U\in \Cnn$ is unitary, $ \Sigma= \textrm{diag} ( \sigma_1 , \dots , \sigma_r ) \in
\Crr $ is a diagonal matrix with the $ r $ positive singular values $ \sigma_1, \dots , \sigma_r $ of $ B $ in its diagonal, $ K \in \Crr $ and $ L \in \Crnr $ satisfy $ K K^* + L L^* = I_r $ (note that $ L $ is absent when $ r = n $). It is worth mentioning that this decomposition always exists but it is not necessarily unique, and that $ B \in \mathbb{C}_1^n $ if and only if $ K $ is nonsingular.

The main aim of this paper is to study lattice properties of the sharp partial order. Malik, Rueda and Thome in \cite{MaRuTh} characterized the matrices that are below a given matrix $B$ under the sharp partial order by using a type (\ref{hartwigB}) decomposition  of $B$. In that work, it is proved that there exists a bijection between the set of matrices that are below $B$ and a certain set of projectors of the same size as $\Sigma $ that commute with $\Sigma K$. In this paper, we prove that this bijection is an isomorphism by considering the set of projectors ordered by the partial order $T_1\leq T_2$ if and only if $T_1= T_1T_2= T_2T_1$. This isomorphism is our starting point for the development of the paper. It brings significant advantages not only because the sizes of the matrices associated with projectors can be considerably smaller than the matrices that are below $B$ but also because these projectors provide us a much easier way to work than by using directly the original matrices.

Many authors worked with projectors to characterize partial orders. This shows the importance of them in finite and infinite dimensions. See, for example,
\cite{CiRuThVe,Marvot2,Marovt3,MiBhMa,Semrl}.

We first establish some lattice properties of the down-set of a fixed matrix $B$, by using this isomorphism. We start by computing the infimum and the supremum of two matrices in $[O, B]$ with commuting associated projectors. We prove that if the down-set is a lattice then it is complemented of finite height and if it is, in addition, a distributive lattice then it is a Boolean algebra. We also show that any interval $[A_1, A_2]=\{A\in\mathbb{C}_1^{n}: A_1\stackrel{_\sharp}\leq  A \stackrel{_\sharp}\leq A_2 \}$, for $A_1\stackrel{_\sharp}\leq  A_2$, of the poset $ \mathbb{C}_1^{n}$ is isomorphic to a down-set, namely $[O, A_2-A_1]$. This is why studying the down-sets are not only important by themselves, but also contributes to the study of the poset $\mathbb{C}_1^{n}$.

On the other hand, we established another isomorphism between the set of projectors that commute with $\Sigma K$ and the set of projectors that commute with the Jordan canonical form of $\Sigma K$. This isomorphism becomes very useful to study, in a more detalied way, the matrices that are below $B$. This tool allows us to characterize when the down-set of $B$ is a lattice. In such a case, we describe it completely and establish some properties of it. We also show that every down-set of $B$ has a distinguished Boolean subalgebra and we give a description of its elements.

In \cite[Theorem 4.3.13]{MiBhMa} Mitra, Bhimasankaram and Malik give a characterization of the matrices $A$ that are below a given matrix $B$ by means of the Jordan decomposition of $B$, for the particular case of the geometric multiplicity of each nonnull eigenvalue equals 1. They also formulate the conjecture that states that the theorem given is valid even if the number of Jordan blocks corresponding to some or all distinct nonnull eigenvalues is greater than 1. We show that this conjecture is not true. In addition, we present a result that characterizes the matrices that are above a given matrix $A$ in terms of the Jordan decomposition of the matrix $A$. We also prove that the set of all $n \times n$ complex matrices having index at most 1 is a lower semilattice if and only if $n\leq 2$. In this way, we extend to $n=3$ the result demonstrated by Mitra in \cite{Mi} which states that this poset is not a lower semilattice for $n\geq 4$.

In \cite{FeLeTh} the authors presented necessary and sufficient conditions for the $k$-commutative equality $B^kX=XB^k$, where $X$ is an outer generalized inverse of the square matrix $B$. Similar matrix equations of type $B^kX^kB^k=B^k$ and $X^kB^kX^k=X^k$, for all $k$ positive integer and where $B\in \mathbb{C}_1^n$ were studied previously in \cite{RaoMi}. As an application of the isomorphism between the set of matrices that are below $B$ and  the set of projectors that commute with $\Sigma K$, we analize certain matrix systems that have equalities $B^kX=XB^k$ and $XB^kX=XB^k$, for all $k$ positive integer, and where $X$ is idempotent. 

This paper is organized as follows. After a brief section of preliminaries, we devote the next two sections to the study of the down-set $[O, B]$ by showing that such a set is isomorphic to two posets of projectors that satisfy some conditions linked to the matrix $\Sigma K$  obtained in decomposition (1). In Section \ref{sec:down-setofsharp}, we work with the poset of projectors commuting with $\Sigma K$ and in Section \ref{sec:deltasharp} with that of projectors commuting with the Jordan canonical form of $\Sigma K$. We also present here several results including the characterization of $[O, B]$ as a lattice.
In Section \ref{sec: Boolean algebra}, we prove the existence of a featured Boolean algebra within the down-set of a matrix $B$ and describe its elements. In Section \ref{sec:posetc1n}, we solve the conjecture formulated in \cite{MiBhMa}. We characterize matrices that are above a given matrix by using its Jordan decomposition, and we analize when $\mathbb{C}_1^n$ is a lower semilattice. Finally, in Section \ref{sec:matrixeq}, we provide solutions to some matrix equations by using the isomorphism given in Section \ref{sec:down-setofsharp}. 

\section{Preliminaries}\label{sec:preliminares}

For the sake of completeness, we recall some basic definitions of structures defined over a partially ordered set that are used throughout this paper. Recall that a partially ordered set (a poset from now on) $(R,\leq)$ is a lower semilattice if for every $a, b\in R$ the greatest lower bound (or infimum) $a\wedge b$ of $\{a, b\}$ exists. The poset $(R,\leq)$ is a lattice if for every $a, b\in R$ both the least upper bound (or supremum) $a\vee b$ and the infimum $a\wedge b$ of $\{a, b\}$ exist. Given two lattices $R$ and $S$, the direct product $R\times S=\{\langle a, b\rangle: a\in R \text{ and } b\in S\}$ is also a lattice where $\langle a_1, b_1\rangle\vee \langle a_2, b_2\rangle= \langle a_1\vee a_2, b_1\vee b_2\rangle$ and $\langle a_1, b_1\rangle\wedge \langle a_2, b_2\rangle= \langle a_1\wedge a_2, b_1\wedge b_2\rangle$. Two elements $a,b$ of a lattice with first element $0$ and greatest element $1$ are complementary if $a\vee b=1$ and $a\wedge b=0$. A complemented lattice is a lattice in which every element $a$ has a complement $a'$. A distributive lattice is a lattice which satisfies either (and hence, as it is easy to see, both) of the distributive laws $a\wedge (b\vee c)= (a\wedge b)\vee (a\wedge c)$ or $a\vee (b\wedge c)= (a\vee b)\wedge (a\vee c)$. Finally, a Boolean algebra is a complemented distributive lattice. Let $ R$ and $ S $ be two posets. A map $\phi\colon R\to S$ is order-preserving if $\phi(a)\leq \phi(b)$ whenever $a\leq b$. We say that $R$ and $S$ are isomorphic if there exists a bijection $\phi$ from $R$ to $S$ such that both $\phi$ and $\phi^{-1}$ are order-preserving. In that case, $\phi$ is called an isomorphism and we write $R \simeq S$. We refer the reader to \cite{Sanka} for more information about the different structures defined above.

If $A\in \mathbb{C}_1^n$ then there exists a (unique) square matrix $X$ that satisfies $AXA=A$, $XAX=X$, and $AX=XA$. This matrix $X$ is called the group inverse of $A$ and it is denoted by $A^\sharp$. For a given matrix $A\in\Cmn$, there exists a unique matrix $ X \in \Cnm $ such that $AX$ and $XA$ are Hermitian, $AX A = A$, and $X A X = X$, which is called the Moore-Penrose inverse of $ A $ and it is denoted by $ A^\dag $.

For $A, B\in \mathbb{C}_1^{n}$, let us recall that the sharp partial order is defined by: $A \stackrel{_{\sharp}}\leq B $ if and only if $ A^\sharp A= A^\sharp B$ and $AA^\sharp = BA^\sharp $, which is equivalent to $A^2=AB=BA$. 

 For a fixed matrix $B$, by using decomposition (\ref{hartwigB}), the predecessors $A$ under the sharp partial order in $\mathbb{C}_1^{n}$ are characterized as follows.

\begin{theorem}\cite[Theorem 5]{MaRuTh}\label{teorema:predecesor} Let $B\in \mathbb{C}_1^{n}$ be a nonzero matrix written as in (\ref{hartwigB}). The following conditions are equivalent. 
\begin{enumerate}[(a)]
 \item There exists a matrix $A\in \mathbb{C}_1^n$  such that 
 $A\stackrel{_\sharp}\leq B$,
 \item There exists an idempotent matrix $T\in \Crr$ such that 
 
    \begin{equation}\label{hartwigAgrupo}
    A = U \begin{bmatrix}
    T\Sigma K  &  T \Sigma L \\
    O & O \\
    \end{bmatrix} U^*,
    \end{equation}

where $T\Sigma K = \Sigma K T$.

\end{enumerate}

\end{theorem}

\begin{remark} For each $A\in \mathbb{C}_1^n$ such that $A\stackrel{_\sharp}\leq B$, it is easy to see that there exists a unique matrix $T \in \Crr$ in the conditions indicated in Theorem \ref{teorema:predecesor}.
 
\end{remark}

We recall that if $A\in \mathbb{C}_1^{n} $ is a matrix written as in (\ref{hartwigAgrupo}) (see \cite[Lemma 3]{MaRuTh}) then 

\begin{equation}\label{inversa grupo}
A^{\sharp} =U \begin{bmatrix}
(T\Sigma K)^{\sharp}  & (T\Sigma K)^{\sharp} K^{-1} L  \\
O & O \\
\end{bmatrix}U^*. 
\end{equation}
Moreover, if $T$ is idempotent and $T\Sigma K=\Sigma K T $ then $(T\Sigma K)^{\sharp}= (\Sigma K)^{-1}T$.

\begin{remark} If $A\stackrel{_\sharp}\leq B$ then, by \cite{BaTr}, $B^\sharp-A^\sharp=U \begin{bmatrix}
(\Sigma K)^{-1}(I_r-T)  & (\Sigma K)^{-1}(I_r-T)K^{-1} L  \\
O & O \\
\end{bmatrix}U^*= (B-A)^{\sharp}$ (see \cite[Theorem 6.2]{ShSi}).
 
\end{remark}

For a given matrix $M\in \Crr$ with $s$ distinct eigenvalues $\{\lambda_1, \lambda_2, \ldots, \lambda_s\}$,
the well-known Jordan decomposition theorem states that there exists a nonsingular matrix $P\in \Crr$ such that $M = PJP^{-1} $ and \begin{equation} \label{Jordan Form}
J= \textrm{diag} (J(\lambda_1), J(\lambda_2), \cdots, J(\lambda_s)),
\end{equation} 
is a block diagonal matrix called the Jordan canonical form of $M$. Here each block $J(\lambda_j)$, for $1\leq j\leq s$, is given by
\begin{equation} \label{Jordan Form 2}
J(\lambda_j) = \textrm{diag}(J_1(\lambda_j), J_2(\lambda_j), \dots, J_{t_j}(\lambda_j))   \in \Crjrj \end{equation}
with
$$J_k(\lambda_j) =  \begin{bmatrix}
\lambda_j & 1 &  &  \\
 &\ddots  & \ddots &  \\
 & &\ddots & 1 \\
 &  & & \lambda_j
\end{bmatrix} \in \Crkjrkj , \quad  1 \leq k \leq t_j.
$$ 
The matrices $J_k(\lambda_j)$ are the Jordan blocks of $J$. We recall that the geometric multiplicity of $\lambda_j$ corresponds to the number of Jordan blocks of $J$ whose eigenvalue is $\lambda_j$, that is the geometric multiplicity of $\lambda_j$ is equal to $t_j$. 

Next, we need the concept of regular upper triangular matrix \cite{Cu}.
\begin{definition}
A matrix $X\in\Crr$ is called a regular upper triangular matrix (RUTM) if it is upper triangular and all the elements
 on any given subdiagonal are equal. This matrix is also known as an upper triangular Toeplitz matrix.
\end{definition}

\begin{remark} \label{idempotente RUTM}
Notice that if
\begin{equation}\label{nilpotente_cero}
N =  \begin{bmatrix}
0 & 1 &  &  \\
 &\ddots  & \ddots &  \\
 & &\ddots & 1 \\
 &  & & 0
\end{bmatrix} \in \Crr
\end{equation}
 and $X$ is a RUTM then $X = a_1 I_r + a_2 N + a_3 N^2 +\dots+ a_{r} N^{r-1}$ for some $a_1,a_2,\dots,a_r \in {\mathbb C}$. 
In addition, if $X^2=X$ then $X=O$ or $X=I_r$. Indeed, $X^2= 
a_1^2 I_r + 2a_1a_2 N + \dots+ (\sum_{i=1}^{l+1} a_ia_{l+2-i})N^l+ \dots + (\sum_{i=1}^{r} a_ia_{r+1-i})N^{r-1}$, and comparing coefficients it follows that
$a_1\in \{0,1\}$ and $a_i=0$, $2\leq i\leq r$.
\end{remark}

Let $T\in \Crr$ and $J$ be written as in (\ref{Jordan Form}). In \cite{Cu}  the matrices $T$ that commute with a matrix $J$ given in Jordan canonical form are characterized. More precisely, by \cite[Theorem 5.16]{Cu},
\begin{equation}\label{TJ=JT}
TJ=JT \qquad \Leftrightarrow \qquad
T= \textrm{diag} ( D_1 , D_2, \dots , D_s ), 
\end{equation}

\noindent where
\begin{equation}\label{Tsss}
D_j=[R_{ik}]\in \Crjrj , 1\leq i, k\leq t_j
\end{equation}
such that the submatrices $R_{ik}\in \mathbb{C}^{r_{ij}\times r_{kj}}$ satisfy that \[R_{ik}J_k(\lambda_j)= J_i(\lambda_j)R_{ik},\] for 
$J_k(\lambda_j)\in \Crkjrkj$, $J_i(\lambda_j)\in \mathbb{C}^{r_{ij}\times r_{ij}}$. Applying again \cite[Theorem 5.16]{Cu}, we obtain that 
\begin{equation}\label{cullen}
R_{ik} = \left\{\begin{array}{cl}
\left[\begin{array}{c}
X_{ik}\\
O
\end{array}\right] \text{ where } X_{ik} \in \Crkjrkj \text{ is a RUTM}, & \text{if }  r_{ij} >r_{kj}\\
\left[\begin{array}{cc}
O &  X_{ik}
\end{array}\right] \text{ where } X_{ik} \in \mathbb{C}^{r_{ij}\times r_{ij}} \text{ is a RUTM}, & \text{if } r_{ij}< r_{kj} \\
\begin{array}{cc}
X_{ik}
\end{array} \text{ where } X_{ik} \in \Crkjrkj \text{ is a RUTM}, & \text{if } r_{ij}= r_{kj} 
\end{array}
\right. .
\end{equation}

\section{Down-set of $B$ via projectors that commutes with $\Sigma K$}\label{sec:down-setofsharp}

For any fixed matrix $B\in \mathbb{C}_1^{n}$, we will consider the down-set
 \[[O, B]=\{A\in \mathbb{C}_1^{n} : O\stackrel{_{\sharp}}\leq A \stackrel{_{\sharp}}\leq B\}.\] 
Clearly, $\left([O, B],\stackrel{_{\sharp}}\leq\right)$ is a poset where $O$ is the least element and $B$ the greatest one. 

According to Theorem \ref{teorema:predecesor}, we consider the set 
     
\[\tau= \{ T \in \Crr : T^2 = T\ \textrm{and } \ T\Sigma K = \Sigma K T\} \] 
\noindent endowed with the partial order given by $T_1\leq T_2 $ if and only if $T_1=T_1T_2=T_2T_1$. Note that if $T\in \tau$ then $T\in\mathbb{C}_1^r $. The least element and the greatest element of $\tau$ are $O$ and $I_r$, respectively. 

In this section, we prove that the down-set $[O, B]$ is isomorphic to the poset $\tau$, which is our starting point for the forthcoming results of this paper. From this isomorphism, we begin the investigation of the ordered structure of $[O, B]$ establishing in this section some lattice properties of the down-set of $B$. Note that matrices $T\in\tau $ are projectors that belong to ${\mathbb C}^{r \times r}$ (instead of ${\mathbb C}^{n \times n}$), with $ 0 < r \leq n, $ where $ r $ can be  considerably smaller than $ n $. So, working with the matrices $T\in\tau $ is easier than using the matrices $A$.

In order to prove the isomorphism, we define the map \[\phi\colon [O, B]\to \tau\] by $\phi(A)=T$ where $T$ is the corresponding projector from Theorem \ref{teorema:predecesor}. As for each $A$, the matrix $T$ is unique, we can assure that $\phi $ is well-defined and it is easy to prove that $\phi$ is a bijection.

\begin{theorem}\label{isomorfismo}Let $B\in \mathbb{C}_1^{n}$ be a nonzero matrix written as in (\ref{hartwigB}). The posets $ [O, B] $ and $ \tau $ are isomorphic. Moreover, the rank function is preserved under the isomorphism $\phi$.
\end{theorem}

\begin{proof} Let $A_i\in [O, B]$ and $\phi(A_i)=T_i$ for $i\in \{1,2\}$. We want to see that $ A_1 \stackrel{_\sharp}\leq A_2 $ iff $T_1\leq T_2$. Since  
\begin{equation*}
    A_i = U \begin{bmatrix}
    T_i\Sigma K  &  T_i \Sigma L \\
    O & O \\
    \end{bmatrix} U^*,
    \end{equation*}
from $T_1\in\tau $ and the expression of the group inverse (\ref{inversa grupo}), we have
\begin{equation*}\label{ecuacioninversasharp}
A_1^{\sharp} =U \begin{bmatrix}
(\Sigma K)^{-1}T_1  & (\Sigma K)^{-1}T_1 K^{-1} L  \\
O & O \\
\end{bmatrix}U^*.
\end{equation*}

\noindent Thus, 
$A_1^\sharp A_1 = A_1^\sharp A_2$ iff $ \begin{bmatrix}
T_1  &  (\Sigma K)^{-1}T_1\Sigma L  \\
O & O \\
\end{bmatrix}= \begin{bmatrix}
T_1T_2  &  (\Sigma K)^{-1}T_1T_2\Sigma L  \\
O & O \\ 
\end{bmatrix}$ iff $T_1= T_1T_2$. Analogously, $A_1A_1^\sharp = A_2A_1^\sharp$ iff $\begin{bmatrix}
T_1  &  T_1K^{-1} L  \\
O & O \\
\end{bmatrix}= \begin{bmatrix}
T_2 T_1  &  T_2T_1K^{-1} L  \\
O & O \\ 
\end{bmatrix}$ iff $T_1= T_2T_1$. So, $ A_1 \stackrel{_\sharp}\leq A_2 $ iff $T_1\leq T_2$.
Hence, $[O, B] \simeq \tau$.

The proof that $\phi$ preserves the rank follows as in \cite[Theorem 3]{CiRuSaTh}.
\end{proof}

\begin{lemma}\label{lem:supinfcuandoconmutan} Let $T_1, T_2\in \tau$ be such that $T_1T_2=T_2T_1$. Then $T_1\wedge T_2$ and $T_1\vee T_2$ both exist, and, $T_1\wedge T_2= T_1T_2$ and $T_1\vee T_2= T_1+T_2-T_1T_2$.    
 
\end{lemma}
\begin{proof} From the assumptions, it is easy to show that $(T_1T_2)^2= T_1T_2$ and $(T_1T_2)\Sigma K = \Sigma K (T_1T_2) $. So, $T_1T_2\in \tau$. Moreover, if $\widetilde{T}\in \tau$ is such that $\widetilde{T}\leq T_1$ and $\widetilde{T}\leq T_2$ then $\widetilde{T}(T_1T_2)=
    (\widetilde{T}T_1)T_2=\widetilde{T}T_2=\widetilde{T} =T_1\widetilde{T} =
    T_1(T_2\widetilde{T})= (T_1T_2)\widetilde{T}$. So, $\widetilde{T} \leq T_1 T_2$, hence $T_1T_2=T_1\wedge T_2$.

If $T_1$ and $T_2$ commute then $(T_1+T_2-T_1T_2)^2= T_1+T_2-T_1T_2$ and $(T_1+T_2-T_1T_2)\Sigma K= \Sigma K(T_1+T_2-T_1T_2)$. So $T_1+T_2-T_1T_2\in \tau$. Let $\widetilde{T}\in \tau$ be such that $T_1\leq \widetilde{T}$ and $T_2\leq \widetilde{T}$. Then, 
\[(T_1+T_2-T_1T_2)\widetilde{T}=T_1 \widetilde{T}+T_2\widetilde{T}-T_1T_2\widetilde{T}= T_1+T_2-T_1T_2=\widetilde{T}(T_1+T_2-T_1T_2). \]
Therefore, $T_1+T_2-T_1T_2=T_1\vee T_2$.
\end{proof}

As an immediate consequence of the above result and the fact that $\phi$ is an isomorphism we have the following result.

\begin{corollary}Let $A_1, A_2 \in [O, B]$ be written as in (\ref{hartwigAgrupo}) such that $T_1T_2 = T_2T_1$, where $ T_i = \phi ( A_i ) $ for every $ i \in \{ 1 , 2 \} $. Then:
\begin{enumerate}[(a)]
   \item $ A_1 \wedge A_2 =A_1B^\dag A_2$,
  
    \item $ A_1\vee A_2 = A_1 + A_2 - A_1 \wedge A_2 $.
\end{enumerate}
 \end{corollary}
\begin{proof}
To see $(a)$, 
 \begin{equation*}
\begin{split}
 & A_1B^\dag A_2  = U \begin{bmatrix} T_1\Sigma K & T_1 \Sigma L \\ O & O \end{bmatrix} \begin{bmatrix} K^*\Sigma^{-1} & O \\ L^*\Sigma^{-1} & O \end{bmatrix}\begin{bmatrix} T_2\Sigma K & T_2 \Sigma L \\ O & O \end{bmatrix} U^*= U \begin{bmatrix} T_1T_2\Sigma K & T_1T_2 \Sigma L \\ O & O \end{bmatrix} U^*\\  & = A_1\wedge A_2.
\end{split}
\end{equation*}

Part $(b)$ follows directly from Lemma \ref{lem:supinfcuandoconmutan}. \qedhere

\end{proof}

\begin{remark}\label{remarklongitudcadena} Assume $ A \overset{\sharp}{\leq} B$ and  $A \neq B$. It is easy to see that $\operatorname{rk} ( A ) < \rk (B)$. From this, we deduce the following: 
\begin{enumerate}[(a)]
 \item If $\rk ( A ) =\rk ( B )$  then $A$ and $B$ are incomparable. 
\item The largest subchain in $[O,B]$ has at most $\rk ( B )+1$ elements.
\end{enumerate}
\end{remark}

\begin{proposition}\label{prop:complementadogrupo} If $[O, B]$ is a lattice then it is complemented of finite height.
\end{proposition}
\begin{proof}
 Let $T\in \tau$. Let us see that $I_r-T\in\tau$. Indeed, it is clear that $(I_r-T)^2=I_r-T$. Since $T\Sigma K = \Sigma K T$, then $(I_r-T)\Sigma K = \Sigma K (I_r-T)$. So $I_r-T\in\tau$. Since  $T(I_r-T)=(I_r-T)T=O$ then $T\wedge (I_r-T)=O$ and $T\vee (I_r-T)=I_r$ by Lemma \ref{lem:supinfcuandoconmutan}. 
\end{proof}

Let us observe that the complement of an element is not necessarily unique. However, if the lattice is distributive then the complements are unique and the following result is satisfied. 

\begin{corollary}\label{cor:distributiveisboole} If $[O, B]$ is a distributive lattice then $[O, B]$ is a Boolean algebra. 
\end{corollary}

In the last part of the section, we will prove that any interval is a copy of a certain down-set, more specifically, the interval $[A_1, A_2]=\{A\in\mathbb{C}_1^{n}: A_1\stackrel{\sharp}\leq A\stackrel{\sharp}\leq A_2 \}$, for any $A_1\stackrel{\sharp}\leq A_2\in \mathbb{C}_1^{n}$ is isomorphic to $[O, A_2-A_1]$. 

\begin{remark}\label{remarkparalema} Let $P, T, Q\in \Crr $ be projectors such that $P\leq T$ and $TQ=QT=O$. It is easy to see that $PQ=QP=O$. 
\end{remark}

\begin{lemma}\label{Lemma intervalos}

Let $A, A_1, A_2, B\in \mathbb{C}_1^{n}$. If $A_1\stackrel{\sharp}\leq A_2\stackrel{\sharp}\leq B$ then $[A_1, A_2]$ and $[O, A_2-A_1]$ are isomorphic. In particular, if $A\stackrel{\sharp}\leq B$ then $[O, B-A]$ and $[A, B]$ are isomorphic.
\end{lemma}
\begin{proof} 

Assume that $ A_1\stackrel{\sharp}\leq A_2\stackrel{\sharp}\leq B$ and set $T_1, T_2$ such that $\phi(A_i)=T_i$, for each $i\in \{1,2\}$. 
If $P$ satisfies $P^2=P\leq T_2-T_1$, by $(T_2-T_1)T_1=O=T_1(T_2-T_1)$ and Remark \ref{remarkparalema}, we have that $PT_1=T_1P=O$. It is easy to see that $P+T_1$ is idempotent and $T_1\leq P+T_1$. Now, again from $P\leq T_2-T_1$, we have that $P=P(T_2-T_1)=PT_2$ and $P=(T_2-T_1)P=T_2P$. Then, $(P+T_1)T_2=PT_2+T_1T_2=P+T_1=T_2P +T_2P_1= T_2(P+T_1)$, i.e., $P+T_1\leq T_2$. Thus, the map $\varphi\colon [O, T_2-T_1]\to [T_1, T_2]$ given by $\varphi(P)=P+T_1$ is well-defined. Let us prove that $\varphi $ is an isomorphism. Indeed, let $Q\in [T_1, T_2]$ and $P=Q-T_1$. Then $P(T_2-T_1)=(Q-T_1)(T_2-T_1)= QT_2-T_1T_2-QT_1+T_1^2= Q-T_1-T_1+T_1=Q-T_1=P$. Similarly, $(T_2-T_1)P=P$. So, $P\in [O, T_2-T_1]$ and $\varphi(P)=Q$. Thus, $\varphi $ is surjective. Let $P_1, P_2\in [O, T_2-T_1]$. 
Since $T_1P_2=P_1T_1=T_1P_1=P_2T_1=O$ then,
$\varphi(P_1)\leq \varphi(P_2)$ if and only if $P_1P_2=P_2P_1=P_1$ if and only if $P_1\leq P_2$. Then, $\varphi$ is an isomorphism.

The second statement follows by setting $A_1=A$ and $A_2=B$. \qedhere
\end{proof}

\section{Down-set of $B$ via projectors that commutes with the Jordan canonical form of $\Sigma K$}\label{sec:deltasharp}

Let $B\in \mathbb{C}_1^{n}$ be written as in (\ref{hartwigB}) and the Jordan decomposition of the nonsingular matrix $\Sigma K$, that is, $\Sigma K = PJP^{-1}$ where the Jordan canonical form of $\Sigma K$ is $ J = \textrm{diag} (J(\lambda_1), J(\lambda_2), \cdots, J(\lambda_s))$ as in (\ref{Jordan Form}) and $s$ is the number of distinct eigenvalues of $\Sigma K$. 

In the previous section we considered the poset $\tau= \{ T \in \Crr : T^2 = T\ \textrm{and } \ T\Sigma K = \Sigma K T\}$ endowed with the partial order given by $T_1\leq T_2 $ if and only if $T_1=T_1T_2=T_2T_1$, and we proved that this poset is isomorphic to the down-set of $B$. Now, we introduce a new poset \[\delta=\{T\in\Crr : \, T^2=T \text{ and } TJ=JT\},\]  also ordered by $\leq  $. In this section we establish an isomorphism between the poset $\tau$ and the poset $\delta$. Clearly, we will have an isomorphism between the down-set of $B$ and the poset $\delta$ given by the composition of both isomorphisms. Consequently, we can study the down-set of $B$ through the poset $\delta$. This brings significant advantages, not only because elements in $\delta$ are projectors (with sizes smaller than those of the corresponding matrices which are in the down-set of $B$) but also because, as these projectors commute with the Jordan canonical form of $\Sigma K$, we can describe them with more precision. So, we get a deeper description of the matrices that are below a given matrix $B$. From this, we obtain conditions under which the down-set of $B$ is a lattice and we completely describe this lattice. We also analize when it is distributive and we give conditions for the down-set of $B$ to be a Boolean algebra. 

In the following result we will prove that $\tau$ is isomorphic to $\delta$. To do this, we define the map $\psi: \delta \to \tau$ by $\psi(T)= PTP^{-1}$, where $P$ is the nonsingular matrix obtained in the Jordan decomposition of $\Sigma K$.

\begin{proposition}\label{isomorfismo-con-delta} The posets $\tau$ and $\delta$, ordered by $\leq $ in both cases, are isomorphic.
\end{proposition}
\begin{proof}

     Is clear that $[\psi (T)]^2=\psi (T)$ because $T^2=T$ and
    $\psi (T)\Sigma K =  \Sigma K\psi (T)$ holds from $TJ=JT$.
    Then $\psi (T)\in \tau$. Hence, $\psi$ is well defined.
    If $T\in\tau$, we have that $P^{-1}T P\in\delta$ and $\psi(P^{-1}T P)=T$, that is $\psi$ is onto.
    Finally, for all $T_1, T_2\in \delta$ we have that $T_1 \leq  T_2$ iff $ 
    T_1= T_1T_2= T_2T_1$ iff $PT_1P^{-1}=P T_1P^{-1}PT_2P^{-1}=P T_2P^{-1}PT_1P^{-1}$ iff  $
    \psi (T_1)=\psi (T_1)\psi (T_2)=\psi (T_2)\psi (T_1)$ iff $ \psi (T_1) \leq
    \psi (T_2) $. \qedhere
\end{proof}

It is easy to see that the rank function is preserved under $\psi$.

Let $M$ be a given square matrix. We consider the poset $\delta_M$ of all the projectors with the same size of $M$ and commuting with $M$, that is,  
\[\delta_M=\{T\colon T^2=T \text{ and } TM=MT\}\]
ordered by $\leq$. Obviously, $\delta=\delta_{J}$.

We start by studying the poset $\delta_J$ in the special case in which $J$ is a Jordan matrix with only one eigenvalue having geometric multiplicity equal to 2. In the following lemma, we determine the possible ranks of a projector in this $\delta_J$. This case is important because it will be our starting point for describing the poset $\delta$, and thus the down-set of $B$, in the general case. 

\begin{lemma}\label{lemmarangos} Let $J=J(\lambda)=\begin{bmatrix}
J_1(\lambda) & O  \\
O & J_2(\lambda)
\end{bmatrix}$ be a block matrix where $J_1(\lambda)$ and $J_2(\lambda)$ are the Jordan blocks associated to the same eigenvalue $\lambda$, $J_1(\lambda)\in \mathbb{C}^{q \times q}$, $J_2(\lambda)\in \mathbb{C}^{p\times p}$, $q\geq p$ and $q+p=r$. Let $T\in \delta_{J}$.
\begin{enumerate}[(a)]
 \item If $q>p$ then $\operatorname{rk} (T)\in \{0, p, q, r\}$. 
 \item If $q=p$ then $\operatorname{rk} (T)\in \{0, r/2 , r\}$. 
\end{enumerate}

\end{lemma}
\begin{proof} 

Since $TJ=JT$, particularizing in (\ref{TJ=JT}) and (\ref{Tsss}), we get $T=D_1=\begin{bmatrix}
R_{11} & R_{12} \\
R_{21} & R_{22}
\end{bmatrix}$, where $R_{11}\in \mathbb{C}^{q\times q}$, $R_{22}\in \mathbb{C}^{p\times p}$, $R_{12}\in \mathbb{C}^{q\times p}$ and $R_{21}\in \mathbb{C}^{p\times q}$ . The equality $T^2=T$ implies  
\begin{equation}\label{eqTcuadrado}
 \begin{bmatrix}
R_{11}^2+ R_{12}R_{21} & R_{11}R_{12}+R_{12}R_{22} \\
R_{21}R_{11}+R_{22}R_{21} & R_{22}^2+ R_{21}R_{12}\end{bmatrix}=\begin{bmatrix}
R_{11} & R_{12} \\
R_{21} & R_{22}
\end{bmatrix}.
\end{equation}

\begin{enumerate}[(a)]
 \item If $q>p$ then by (\ref{cullen}) $T= \left[\begin{array}{c|c}
X_{11} & \begin{matrix} X_{12} \\ O
          \end{matrix} \\ \hline
 \begin{matrix}
O & X_{21}  \end{matrix}
 & X_{22}
\end{array}\right] $ where $X_{11}\in \mathbb{C}^{q\times q}$, $X_{12}, X_{21}, X_{22}\in \mathbb{C}^{p\times p}$ are RUTM's. Firstly, we observe that $ R_{12}R_{21}=\begin{bmatrix} X_{12} \\ O
          \end{bmatrix}  \begin{bmatrix}
O & X_{21}  \end{bmatrix}= \begin{bmatrix} O & X_{12}X_{21} \\
O & O\end{bmatrix}
$ and then the main diagonal has all the elements equal 0. Secondly, we have that $ R_{21}R_{12}=\begin{bmatrix}
O & X_{21}  \end{bmatrix}\begin{bmatrix} X_{12} \\ O
          \end{bmatrix} $ and easy computations yields to the elements in the main diagonal of $ R_{21}R_{12}$ are again all equal $0$. 
          
Since $R_{11}= X_{11}$ is a RUTM, the elements in the main diagonal of $R_{11}$ are all equal to certain $a$. From the above and (\ref{eqTcuadrado}) we have that $a=a^2$. So, $a=0$ or $a=1$. Similarly, if we call $b$ the elements in the main diagonal of $R_{22}= X_{22}$ we obtain that $b=0$ or $b=1$. 

Now, we argue by computing ranks. Let us recall that if $T$ is a projector then $\operatorname{rk} (T)$ is equal to the trace of $T$ which is denoted by $\operatorname{tr}(T)$. Thus, if $a=b=0$ then $T=O$ and if $a=b=1$ then $T=I_{r}$. If $a=1$ and $b=0$ then $\operatorname{rk} (T)= q$ and if $a=0$ and $b=1$ then $\operatorname{rk} (T)= p$. 

\item If $q=p$, then $r=2q$ and then by (\ref{cullen}) $T=\begin{bmatrix}
X_{11} & X_{12} \\
X_{21} & X_{22}
\end{bmatrix}$ where $X_{ij}\in\mathbb{C}^{q\times q}$ is a RUTM, for $i,j\in \{1,2\}$. Let $a, b, c, d$ be the elements of the main diagonal of $X_{11}$, $X_{22}$, $X_{12}$, and $X_{21}$, respectively. Using (\ref{eqTcuadrado}) we have $a^2+cd=a$ and $b^2+cd=b$. Thus, $(a-b)(a+b)=a-b$. If $a\neq b$, then $a=1-b$ and $\operatorname{rk}(T)= \operatorname{tr} (T)= q=r/2$. If $a=b$ then $\operatorname{rk}(T)= \operatorname{tr} (T)= ra $. If $a=1$ then $T=I_{r}$ and if $a=0$ then $T=O$. If $a\neq 0$ and $a\neq 1$ then $\operatorname{rk}(X_{11})=r/2$ and this implies that $\operatorname{rk}(T)\geq r/2$. Similarly, $\operatorname{rk}(I_r-T)\geq r/2$. Since $\operatorname{rk}(T)+ \operatorname{rk}(I_r-T)=r$ then $\operatorname{rk}(T) =r/2$. \qedhere
 \end{enumerate}
\end{proof}

\begin{remark}\label{remark:diagonal} From the proof of Lemma \ref{lemmarangos} we also obtain that if $q>p$ and $\operatorname{rk}(T)=q$ then the elements in the main diagonal of $X_{11}$ are all equal to 1 and the elements in the main diagonal of $X_{22} $ are equal to $0$. In the same way, if $\operatorname{rk}(T)=p$ then the elements in the main diagonal of $X_{11}$ are all equal to $0$ and the elements in the main diagonal of $X_{22} $ are equal to $1$.
\end{remark}

In the next theorem we describe the poset $\delta_J$ in which $J$ is a Jordan matrix as in (\ref{Jordan Form}) having only one eigenvalue and no more than two Jordan blocks associated to it. 

\begin{theorem} Let $J=J(\lambda)\in \mathbb{C}^{r\times r}$ be a Jordan matrix having only one eigenvalue $\lambda$. 
\begin{enumerate}[(a)]
 \item If the geometric multiplicity of $\lambda$ is 1 then $\delta_J=\{O, I_r\}$. 
 \item  If the geometric multiplicity of $\lambda$ is 2 then the projectors in $\delta_J$ distinct from $O$ and $I_r$ are pairwise incomparable. Thus, $\delta_J$ is a lattice with first element, greatest element, and between both of them infinite incomparable elements as shown in the Hasse diagram given in Figure \ref{drank2}. 
\begin{figure}[ht]
\centering
\includegraphics[scale=0.6]{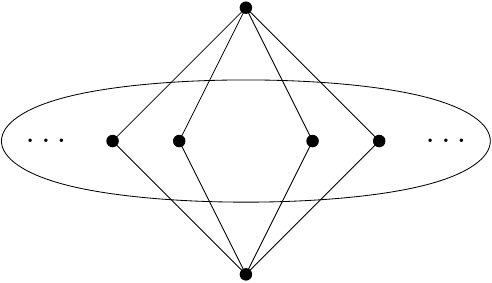}
\caption{$\delta_J$ when $J$ has exactly only one eigenvalue of geometric multiplicity 2.} \label{drank2}
\end{figure}

\end{enumerate}
\end{theorem}
\begin{proof} 
\begin{enumerate}[(a)]
\renewcommand{\arraystretch}{1.5}

 \item Let $T\in \delta_J$. Since $J$ has only one eigenvalue with geometric multiplicity 1 and $TJ=JT$, then $T$ is a RUTM and $T= a_{1} I_{r} + a_{2} N + a_{3} N^2 +\dots+ a_{r} N^{r-1}$ for some $a_{1},a_{2},\dots,a_{r} \in {\mathbb C}$. Moreover, $T=T^2$. Then by Remark \ref{idempotente RUTM}, $a_{1}\in\{0, 1\}$ and $a_{i}=0$ for $2\leq i \leq r$. Thus $\delta_J=\{O, I_r\}$. 
 
\item Let $J$ have exactly two Jordan blocks of sizes $q$ and $p$. Suppose first that $q>p$ and we take $T_1, T_2\in \delta_J\setminus\{O, I_r\}$. If $\operatorname{rk} (T_1)=\operatorname{rk} (T_2)$ then $T_1$ and $T_2$ are incomparable by part (a) in Remark \ref{remarklongitudcadena}. If $\operatorname{rk} (T_1)\neq \operatorname{rk} (T_2)$ then by Lemma \ref{lemmarangos} either $\rk (T_1)=q$ and $ \rk (T_2)=p$, or, $\rk (T_1)=p$ and $ \rk (T_2)=q$. Without lost of generality, we can suppose that $\rk (T_1)=q$ and $ \rk (T_2)=p$. We know that 
 $T_1= \left[\begin{array}{c|c}
X_{11} & \begin{matrix} X_{12} \\ O
          \end{matrix} \\ \hline
 \begin{matrix}
O & X_{21}  \end{matrix}
 & X_{22}
\end{array}\right] $ where $X_{11}\in \mathbb{C}^{q\times q}$, $X_{12}, X_{21}, X_{22}\in \mathbb{C}^{p\times p}$ are RUTM's. Moreover, by Remark \ref{remark:diagonal}, the elements in the main diagonal of $X_{11}$ are all equal to 1 and the elements in the main diagonal of $X_{22}$ are equal to 0. In the same way, $T_2= \left[\begin{array}{c|c}
Y_{11} & \begin{matrix} Y_{12} \\ O
          \end{matrix} \\ \hline
 \begin{matrix}
O & Y_{21}  \end{matrix}
 & Y_{22}
\end{array}\right] $ where $Y_{11}\in \mathbb{C}^{q\times q}$, $Y_{12}, Y_{21}, Y_{22}\in \mathbb{C}^{p\times p}$ are RUTM's with all the elements in the main diagonal of  $Y_{11}$ equal to 0 and all the elements in the main diagonal of $Y_{22}$ equal to 1. It is straightforward to see that the elements in the main diagonal of $T_1T_2$ are all equal to $0$. So, $T_1T_2\neq T_1$ and $T_1T_2\neq T_2$. Thus, $T_1 $ and $T_2$ are incomparable.

If $q=p$ then all $T\in \delta_J\setminus\{O, I_r\}$ have the same rank by Lemma \ref{lemmarangos} and thus they are pairwise incomparable by Remark \ref{remarklongitudcadena}. \qedhere
\end{enumerate}
\end{proof}

As a consequence of the previous theorem and the isomorphism between the poset $\delta=\delta_J$ and the down-set of $B$, we can describe this last poset for the case where the Jordan canonical form $J$ of $\Sigma K$ has only one eigenvalue and no more than two Jordan blocks associated to it.

\begin{theorem}\label{teo:unicoautovalor} Let $B\in \mathbb{C}_1^{n}$ be as in (\ref{hartwigB}) where $\Sigma K$ has only one eigenvalue $\lambda$ and $J=J(\lambda)$ is the Jordan canonical form of $\Sigma K$.
\begin{enumerate}[(a)]
 \item If the geometric multiplicity of $\lambda$ is 1 then $\delta$ and $[O, B]$ are lattices (chains) with only two elements.
 \item  If the geometric multiplicity of $\lambda$ is 2 then $\delta$ and $[O, B]$ are lattices as shown in Figure \ref{drank2}.  
\end{enumerate}
\end{theorem}

Now we analize the general case where $\Sigma K$ has more than one eigenvalue. First, we study the case in which the geometric multiplicity of each eigenvalue is less than or equal to 2.

\begin{theorem}\label{thm:descripciondelta} Let $B\in \mathbb{C}_1^{n}$ be written as in (\ref{hartwigB}) where $\Sigma K$ has $s$ distinct eigenvalues such that the geometric multiplicity of each one of them is less or equal to 2. Then $\delta$ is isomorphic to a finite direct product of lattices that are either chains of two elements or isomorphic to the lattice whose diagram is given in Figure \ref{drank2}.  
\end{theorem}
 \begin{proof} Let $T_i\in \delta$, $i\in \{1,2\}$. Then $T_i= \textrm{diag}(D_{i1}, \cdots, D_{is})$ where each $D_{ij}\in \delta_{J(\lambda_j)}$. Moreover, $T_1\leq T_2$ if and only if $\textrm{diag}(D_{11}, \cdots, D_{1s})\leq \textrm{diag}(D_{21}, \cdots, D_{2s})$ if and only if $D_{1j}\leq D_{2j}$ for each $j\in \{1, \cdots, s\}$.
 Applying Theorem \ref{teo:unicoautovalor} for each $j$, we have that $\delta = \delta_{J(\lambda_1)}\times \cdots \times \delta_{J(\lambda_s)}$ where $\delta_{J(\lambda_j)}$ is either a chain of two elements or the lattice depicted in Figure \ref{drank2}. 
 It is easy to see that $T_1\vee T_2=\textrm{diag} (D_{11}\vee D_{21},  \dots, D_{1s}\vee D_{2s}) $ and $T_1\wedge T_2=\textrm{diag} (D_{11}\wedge D_{21},  \dots, D_{1s}\wedge D_{2s}) $. 
 \end{proof}

\begin{corollary}\label{sharpBoolean} Let $B\in \mathbb{C}_1^{n}$ be decomposed as in (\ref{hartwigB}) where $\Sigma K$ has $s$ distinct eigenvalues and $J$ is the Jordan canonical form of $\Sigma K$. If $J$ has only one Jordan block for each eigenvalue, that is $t_j=1$ in (\ref{Jordan Form 2}) for all $j\in \{1,\dots, s\}$, then $\delta$, and in consequence the down-set $[O, B]$ and $\tau$, are Boolean algebra with $2^s$ elements.
\end{corollary}
 
It is well-known that a lattice $R$ is nondistributive if and only if there is a sublattice of $R$ isomorphic to $M_5$, whose Hasse diagram is given in Figure \ref{M5} (see for example \cite[Theorem 3.6]{Sanka}). So, the next result holds.  
\begin{figure}[ht]
 \centering
\includegraphics[scale=0.7]{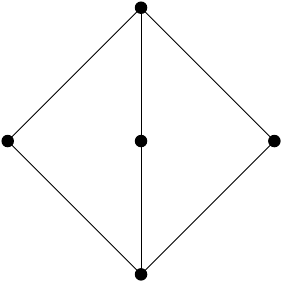}
\caption{The lattice $M_5$}\label{M5}
\end{figure}

\begin{corollary} If $J$ is such that some $t_j=2$ in (\ref{Jordan Form 2}) then $[O, B]$ is nondistributive. 
\end{corollary}

In the next example we show the lattice $[O, B]$ where $\Sigma K$ has two distinct eigenvalues, $\lambda_1$ and $\lambda_2$, such that the geometric multiplicity of $\lambda_1$ is 2 and the geometric multiplicity of $\lambda_2$ is 1.
\begin{example}\label{ejemplolattice} Let $B \in {\mathbb C}_1^{n}$ be such that $\Sigma K= PJP^{-1}$ with
$
J=  \begin{bmatrix}
J(\lambda_1) & O  \\
O & J(\lambda_2)
\end{bmatrix}, \quad
J(\lambda_1)=  \begin{bmatrix}
J_1(\lambda_1) & O  \\
O & J_2(\lambda_1)
\end{bmatrix}, \quad \text{ and } \quad
J(\lambda_2) =J_1(\lambda_2)
$ with $\lambda_1\neq \lambda_2$, that is $t_1=2$ and $t_2=1$. Then $[O, B]$ is the lattice presented in Figure \ref{producto} where $\phi(A_1)=\textrm{diag}(D_1,O)$, $D_1\in \mathbb{C}^{r_1\times r_1}$ a projector as in Lemma \ref{lemmarangos}, $\phi(A_2)=\textrm{diag}(I_{r_1}, O)$, $\phi(A_3)=\textrm{diag}(O, I_{r_2})$ and $\phi(A_4)=\textrm{diag}(D_1,I_{r_2})$.
\begin{figure}[ht]
 \centering
\includegraphics[scale=0.7]{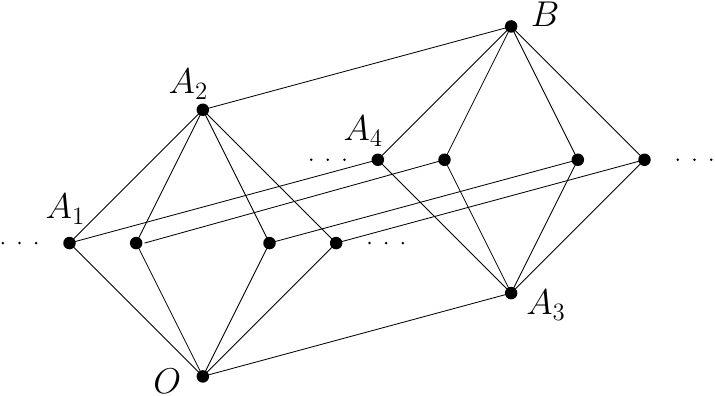}
\caption{$[O, B]$ with $t_1=2$ and $t_2=1$}\label{producto}
\end{figure} 
\end{example}

In the next theorem we give the characterization of $[O, B] $ as a lattice in terms of the Jordan canonical form of $\Sigma K$.

\begin{theorem}\label{teo:tj1o2}  Let $B\in \mathbb{C}_1^{n}$ be written as in (\ref{hartwigB}) and let $J=\textrm{diag}(J(\lambda_1),J(\lambda_2), \dots, J(\lambda_s ))$ be the Jordan canonical form of $\Sigma K$ as in (\ref{Jordan Form}). Then, $[O, B]$ is a lattice if and only if $t_j\in \{1, 2\}$ for all $1\leq j\leq s$.
\end{theorem}
\begin{proof} Let us suppose that $t_j\geq 3$ for some $1\leq j\leq s$. Without loss of generality, we can assume  $t_1\geq 3$. Then, there are at least three Jordan blocks $J_k(\lambda_1)\in \mathbb{C}^{r_{k1}\times r_{k1}}$, for $k\in \{1,2,3\}$ associated to the eigenvalue $\lambda_1$. We can assume also that the sizes of the matrices satisfy that $r_{11}\geq r_{21}\geq r_{31}$.

We will find $T_1, T_2, T_3, T_4\in \delta$ such that neither the infimum of $T_3$ and $T_4$ nor the supremum of $T_1$ and $T_2$ exist (see Figure \ref{fig:nolattice}), which leads to a contradiction.
\begin{figure}[ht]
 \centering
\includegraphics[scale=0.7]{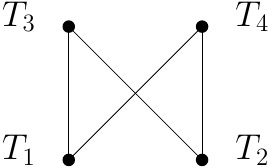}
\caption{Neither $T_1\vee T_2$ nor $T_3\wedge T_4$ exist}\label{fig:nolattice}
\end{figure} 

Let $i\in \{1, 2, 3, 4\}$, $T_i=\textrm{diag}(D_{i1}, D_{i2}, \cdots, D_{is})$ as in (\ref{TJ=JT}) with $D_{ij}=O$ if $j>1$ and $D_{i1}=\begin{bmatrix}
           G_i & O \\                                                                                                                   
            O  & O                                                                                                               \end{bmatrix}
$. Here the size of the $G_i$ is $(r_{11}+r_{21}+r_{31})\times (r_{11}+r_{21}+r_{31})$. 
 
 Let us consider $X=\begin{bmatrix} I_{r_{21}} \\ O                                                                                                                                                                        
                                                                                                                                                        \end{bmatrix}
 \in \mathbb{C}^{r_{11}\times r_{21}}$ and  $Y=\begin{bmatrix} O & I_{r_{31}}                                                                                                                                                                     
                                                                                                                                                        \end{bmatrix}
 \in \mathbb{C}^{r_{31}\times r_{21}}$. It is easy to see that $J_1(\lambda_1) X=XJ_2 (\lambda_1)$ and $J_3(\lambda_1)Y=YJ_2(\lambda_1)$. We also consider $G_1=\begin{bmatrix}
  O & X & O\\                                                                                                      
  O & I_{r_{21}} & O\\
  O & O & O
  \end{bmatrix}
$, $G_2=\begin{bmatrix}
  O & X & O\\                                                                                                      
  O & I_{r_{21}} & O\\
  O & Y & O
  \end{bmatrix}
$ and $G_3=\begin{bmatrix}
  O & X & O\\                                                                                                      
  O & I_{r_{21}} & O\\
  O & O & I_{r_{31}}
  \end{bmatrix}
$. It is straightforward to see that $T_1, T_2, T_3\in \delta$, $T_1\leq T_3$, $T_2\leq T_3$, and $T_1$ is incomparable with $T_2$. Let $T\in \delta$ such that $T_2\leq T\leq T_3$. Then $T_2=T$ or $T_3=T$. Indeed, from $T\leq T_3$ we have that $T= \textrm{diag}(D_1, O, \dots, O)$ where $D_1=\begin{bmatrix}
           D & O \\                                                                                                                   
            O  & O                                                                                                               \end{bmatrix}$ with $D=\begin{bmatrix}
           R_{11} & R_{12} & R_{13}\\                                                                                                                   
            R_{21}  & R_{22} & R_{23}\\
            R_{31} & R_{32} & R_{33}
            \end{bmatrix}$ as in (\ref{cullen}), $R_{kk}\in \mathbb{C}^{r_{k1}\times r_{k1}}$, $1\leq k\leq 3$.  
Since $T_2= T_2T$ then $R_{22}= I_{r_{21}}$, $R_{21}=O$ and $R_{23}=O$. By $T=TT_3=T_3T$ we have that $R_{11}=O$, $R_{31}=O$, $R_{12}=X$ and $R_{13}=O$, that is $D=\begin{bmatrix}
                                                                                                                                                              O & X & O\\
                                                                                                                                                              O & I_{r_{21}} & O\\
                                                                                                                                                              O & R_{32} & R_{33}
                                                                                                                                                             \end{bmatrix}
$. 
From $T_2=TT_2$ we obtain $R_{32}+ R_{33}Y=Y$. Now, using that $T^2=T$ we have that $R_{33}^2= R_{33}$ and $R_{33}R_{32}= O$. Since $R_{33}$ is a RUTM, by Remark \ref{idempotente RUTM} we arrive at $R_{33}=O $ or $R_{33}=I_{r_{31}} $. If $R_{33}=O $ then $R_{32}=Y  $ and $T=T_2$. If $R_{33}=I_{r_{31}} $ then $R_{32}=O  $ and $T=T_3$. Analogously, we can prove that if $T\in \delta$ and $T_1\leq T\leq T_3$, then $T_1=T$ or $T_3=T$.

In order to construct $T_4$ we consider two cases. If $r_{11}=r_{21}$ then $G_4=\begin{bmatrix}
  O & X & O\\                                                                                                      
  O & I_{r_{21}} & O\\
  Z & Y & I_{r_{31}}
  \end{bmatrix} $ where $Z=\begin{bmatrix} 
  O & -I_{r_{31}}
  \end{bmatrix}\in \mathbb{C}^{r_{31}\times r_{11}}$. If $r_{11}>r_{21}$ then $G_4=\begin{bmatrix}
  O & X & O\\                                                                                                      
  O & I_{r_{21}} & O\\
  Z & O & I_{r_{31}}
  \end{bmatrix} $ with $Z=  (z_{ij})\in \mathbb{C}^{r_{31}\times r_{11}}$, where $z_{ij}=\begin{cases}
                                                               -1 & \text { if } i=1 \text{ and } j=r_{11} \\
                                                              0 & \text { otherwise }
                                                              \end{cases}
$. Taking into account that $ZJ_1(\lambda_1)=J_3(\lambda_1)Z$, it is straightforward to see that $T_4\in \delta$, $T_3$ and $T_4$ are incomparable, and that $T_4$ covers $T_1$ and $T_2$. This means that $[O, B]$ is not a lattice. 

The ``if'' part is given by Theorem \ref{thm:descripciondelta}. 
\end{proof}

\section{A distinguished Boolean algebra in the down-set}\label{sec: Boolean algebra}

In this section we prove that there exists a distinguished poset within the down-set of a matrix $B$ that turns out to be a Boolean algebra and describe its elements. We present this algebra in the next Lemma.

\begin{lemma}\label{lemma:subalgebrabooleana}  The poset $\mathcal{C}=\left\{ T\in \tau : T\widetilde{T}=\widetilde{T}T \textrm{ for all } \widetilde{T}\in \tau \right\}$ ordered by $\leq$ is a Boolean algebra.
\end{lemma}
 \begin{proof}
It is clear that $O, I_r\in \mathcal{C}$ and if $T_1, T_2 \in \mathcal{C}$ then $T_1T_2, T_1+T_2-T_1T_2, I_r-T\in \mathcal{C}$. 
By Lemma \ref{lem:supinfcuandoconmutan}, $T_1 \wedge T_2 = T_1T_2$ and $T_1\vee T_2= T_1+T_2-T_1T_2$. Moreover, if $T_3 \in \mathcal{C}$
then $T_1\wedge(T_2\vee T_3)= (T_1\wedge T_2)\vee (T_1\wedge T_3)$. So, $\mathcal{C}$ is a complemented distributive lattice and this implies that 
it is a Boolean algebra.
\end{proof} 

Now we use the isomorphism $\psi\colon \delta\to \tau$ defined by $\psi (T)= PTP^{-1}$ in the previous section to describe the elements in the Boolean algebra. Since $\psi$ is an isomorphism, we know that the set $\psi^{-1}(\mathcal{C})$ is a Boolean algebra within the poset $\delta$. By the definition of $\psi$, it is easy to see that $\psi$ preserves the matrix product. Indeed, 
\[\psi (T_1)\psi (T_2)= PT_1P^{-1}PT_2P^{-1}= PT_1T_2P^{-1}= \psi(T_1T_2).\]

Thus, if $T\in \psi^{-1}(\mathcal{C})$ then $T$ commutes with all the projectors in $\delta$. Moreover, in particular $T\in \delta$ then by the characterization given in (\ref{TJ=JT}) of projectors that commutes with the Jordan canonical form $J = \textrm{diag} (J(\lambda_1), J(\lambda_2), \cdots, J(\lambda_s))$ of $\Sigma K$, we know that $T=  \textrm{diag} ( D_1 , D_2, \dots , D_s ) $ and $D_j= [R_{\alpha\beta}]$ where the submatrices $R_{\alpha \beta}$ have the form given in (\ref{cullen}). In the next lemma we characterize the submatrices $D_j$, for each $j$, and in consequence we obtain a description of the elements in the Boolean algebra $\psi^{-1}(\mathcal{C})$.  

\begin{lemma} Let $\mathcal{C}$ be the Boolean algebra defined in Lemma \ref{lemma:subalgebrabooleana} and $T\in \psi^{-1}( \mathcal{C})$. Then $T= \textrm{diag} ( D_1 , D_2, \dots , D_s ) $ with $D_j\in \{O, I_{r_j}\}$. Therefore, $\psi^{-1}( \mathcal{C})$ has $2^s$ elements, where $s$ is the number of distinct eigenvalues of $\Sigma K$. 
\end{lemma}
\begin{proof} First, we are going to construct a family $\{T_{kj}\} $ of projectors in $\delta$ as follows. For each $j$, $1\leq j\leq s$, and for each $k$, $1\leq k\leq t_j$ consider $T_{kj}= \textrm{diag} ( O, \dots , D_{kj}, \dots, O)  $ with $D_{kj}$ at the $j$-th position and $D_{kj}= [\widetilde{R}_{\alpha\beta}]\in \Crjrj$ as in (\ref{cullen}) where $\widetilde{R}_{\alpha\beta}=\begin{cases}
                I_{r_{kj}} & \text{if } \alpha=\beta=k  \\                                                                                                                                                              O & \text{elsewhere}                                                                                                                                                                                            \end{cases}
 $.  We recall that $r_{kj}\times r_{kj}$ is the size of the Jordan block $J_{k}(\lambda_j)$ of $J(\lambda_j)$. 

Take $T\in \psi^{-1}( \mathcal{C})$. Then $T= \textrm{diag} ( D_1 , D_2, \dots , D_s ) $ where $D_j=[R_{\alpha\beta}]\in \Crjrj$, for each $j\in \{1, \cdots, s\}$, and since $T$ commutes with the projectors of $\delta$, in particular $T$ commutes with the projectors $T_{kj}$. Now, $TT_{kj}=T_{kj}T$ iff $D_jD_{kj}= D_{kj}D_j$, for all $j$, iff $[R_{\alpha\beta}][\widetilde{R}_{\alpha\beta}]=[\widetilde{R}_{\alpha\beta}][R_{\alpha\beta}]$ iff
\begin{equation*}
\begin{split}
 \begin{bmatrix}
R_{11} & \cdots & R_{1k} & \cdots & R_{1t_j}\\
 \vdots & \ddots &  \vdots &  \vdots &  \vdots\\
 R_{k1} & \cdots & R_{kk} & \cdots & R_{kt_j}\\
 \vdots & \vdots &  \vdots &  \ddots &  \vdots\\
R_{t_j1} & \cdots & R_{t_jk} &  \cdots & R_{t_jt_j}
\end{bmatrix} \begin{bmatrix}
O & \cdots & O &  \cdots & O\\
\vdots & \ddots &  \vdots &  \vdots &  \vdots \\
O & \cdots & I_{r_{kj}} &  \cdots & O  \\
\vdots & \vdots &  \vdots &  \ddots &  \vdots\\
O & \cdots & O & \cdots & O
\end{bmatrix} =\begin{bmatrix}
O & \cdots & R_{1k} &  \cdots & O\\
\vdots & \ddots &  \vdots &  \vdots &  \vdots \\
O & \cdots & R_{kk} &  \cdots & O  \\
\vdots & \vdots &  \vdots &  \ddots &  \vdots\\
O & \cdots & R_{t_jk} & \cdots & O
\end{bmatrix}= \\
\begin{bmatrix}
O & \cdots & O &  \cdots & O\\
\vdots & \ddots &  \vdots &  \vdots &  \vdots \\
O & \cdots & I_{r_{kj}} &  \cdots & O  \\
\vdots & \vdots &  \vdots &  \ddots &  \vdots\\
O & \cdots & O & \cdots & O
\end{bmatrix} 
\begin{bmatrix}
R_{11} & \cdots & R_{1k} & \cdots & R_{1t_j}\\
 \vdots & \ddots &  \vdots &  \vdots &  \vdots\\
 R_{k1} & \cdots & R_{kk} & \cdots & R_{kt_j}\\
 \vdots & \vdots &  \vdots &  \ddots &  \vdots\\
R_{t_j1} & \cdots & R_{t_jk} &  \cdots & R_{t_jt_j}
\end{bmatrix} =
\begin{bmatrix}
O & \cdots & O & \cdots & O\\
 \vdots & \ddots &  \vdots &  \vdots &  \vdots\\
 R_{k1} & \cdots & R_{kk} & \cdots & R_{kt_j}\\
 \vdots & \vdots &  \vdots &  \ddots &  \vdots\\
O & \cdots & O &  \cdots & O
\end{bmatrix} 
\end{split}
\end{equation*}

\noindent iff $R_{\alpha k}=O$ and $R_{k\alpha}=O$ for all $\alpha\neq k$, $1\leq k\leq t_j$ and $1\leq \alpha\leq t_j$. Hence, $D_j=\textrm{diag}(R_{11}, R_{22}, \cdots, R_{t_jt_j})$. Since $D_j=D_j^2$ then $R_{kk}= R_{kk}^2$ and by Remark \ref{idempotente RUTM} we have that $R_{kk}\in \{O, I_{r_{kj}}\}$.

If the Jordan canonical form $J$ has only one Jordan block associated with the eigenvalue $\lambda_j$, that is if $t_j=1$, then $D_j=R_{11}$ and, in consequence $D_j\in\{O, I_{r_{j}}\}$. Hence, our claim has been proved in this particular case. 

Now, let us consider that $t_j>1$. Then $D_j=\textrm{diag}(R_{11}, R_{22}, \cdots, R_{t_jt_j})$. Without loss of generality we can assume that the sizes of the submatrices $R_{kk}$, $1 \leq k\leq t_j$, satisfy that $r_{1j}\geq r_{2j} \geq \cdots \geq r_{t_jj}$. To prove our claim, we consider the $j$ projectors in $\delta$ defined as follows $T_j=\textrm{diag}(O, \cdots, \widetilde{D}_j, \cdots, O)$ where $\widetilde{D}_j=\left[\begin{array}{c|ccc}
I_{r_{1j}} & O  & \cdots & O\\
\begin{matrix} O & I_{r_{2j}} 
\end{matrix}
&    O  &  \cdots & O \\

\vdots &  \vdots & \vdots &  \vdots \\
\begin{matrix} O & I_{r_{t_jj}} 
\end{matrix}
&    O &  \cdots & O \\
\end{array}\right] $. Then, $TT_j=T_jT $ if and only if $D_j\widetilde{D}_j=\widetilde{D}_jD_j$ if and only if
\begin{equation*}
\begin{split}
\begin{bmatrix}
R_{11}   & O & \cdots & O\\
O & R_{22} &    \cdots &  O\\
\vdots &  \vdots &  \ddots &  \vdots\\
 O & O &  \cdots & R_{t_jt_j}
\end{bmatrix} 
\left[ \begin{array}{c|ccc}
I_{r_{1j}} & O  & \cdots & O\\
\begin{matrix} O & I_{r_{2j}} 
\end{matrix}
&    O  &  \cdots & O \\
\vdots &  \vdots & \vdots &  \vdots \\
\begin{matrix} O & I_{r_{t_jj}} 
\end{matrix}
&    O &  \cdots & O \\
\end{array}\right] = \\
\left[ \begin{array}{c|ccc}
I_{r_{1j}} & O  & \cdots & O\\
\begin{matrix} O & I_{r_{2j}} 
\end{matrix}
&    O  &  \cdots & O \\
\vdots &  \vdots & \vdots &  \vdots \\
\begin{matrix} O & I_{r_{t_jj}} 
\end{matrix}
&    O &  \cdots & O \\
\end{array}\right]
\begin{bmatrix}
R_{11}   & O & \cdots & O\\
O & R_{22} &    \cdots &  O\\
\vdots &  \vdots &  \ddots &  \vdots\\
 O & O &  \cdots & R_{t_jt_j}
\end{bmatrix}
\end{split}
\end{equation*} 
 if and only if $R_{kk}\begin{bmatrix} O & I_{r_{kj}} 
\end{bmatrix} =\begin{bmatrix} O & I_{r_{kj}} 
\end{bmatrix}R_{11}$ for each $k$, $2\leq k\leq t_j$.  
From this and taking into account that $R_{kk}\in \{O, I_{r_{kj}}\}$,
 $R_{11}=O$ implies that $R_{kk}=O$ for each $k$, and $R_{11}=I_{r_{1j}}$ implies that $R_{kk}=I_{r_{kj}}$ for each $k$.
Therefore $D_j=O$ or $D_j= I_{r_j}$, and this completes the proof.  
 \end{proof}

As an example, if we consider the down-set of $B $ given in Example \ref{ejemplolattice}, then the Boolean subalgebra $\mathcal{C}$ has the elements $\{O,\phi(A_2), \phi(A_3), I_r \}$ depicted in Figure \ref{producto}.

\section{The poset $\mathbb C_1^n$}\label{sec:posetc1n}

In this section we prove that $\left( \mathbb{C}_1^n, \stackrel{_{\sharp}}\leq  \right)$ is a lower semillatice if and only if $n\leq 2$. In order to do that, we establish a characterization of the matrices that are above a given matrix $A$ in terms of its Jordan canonical form.

The Jordan decomposition of a given matrix $B$ provides us a description of matrices $A$ that are below $B$ (\cite{MiBhMa}). More precisally, let $A$ and $B$ be square matrices of the same size. Let $ B\in \mathbb{C}_1^n$ and let $B=P \textrm{diag}\left(J_1, \dots, J_l, O \right)P^{-1}$ be the Jordan decomposition of $B$, where $J_1, \dots, J_l$ are nonsingular Jordan blocks and $P$ is a nonsingular matrix. Let $A= P \textrm{diag}\left(D_1, \dots, D_l, O \right)P^{-1}$ where $D_i\in\{O, J_i\}$. It is easy to check that $A$ is of index less or equal 1 and $A\stackrel{\sharp}\leq B$. The converse of this statement in the special case where there is exactly one Jordan block corresponding to each nonzero eigenvalue is given in \cite[Theorem 4.3.13]{MiBhMa}.

\begin{theorem}\label{caracterizacionlibro}\cite[Theorem 4.3.13]{MiBhMa} Let $A$ and $B$ be square matrices of the same size. Let $B$ be of index less or equal 1 and $B= P \textrm{diag}\left(J_1, \dots, J_l, O \right)P^{-1}$ be the Jordan decomposition of $B$, where $J_1, \dots, J_l$ are nonsingular Jordan blocks corresponding to distinct eigenvalues $\lambda_1, \dots, \lambda_l$ and $P$ is a nonsingular matrix. Then $A$ is of index less or equal 1 and $A\stackrel{\sharp}\leq B$ if and only if $A= P \textrm{diag}\left(D_1, \dots, D_l, O \right)P^{-1}$ where $D_i\in\{O, J_i\}$. 
\end{theorem}

In \cite{MiBhMa} the authors state the following conjecture:
\paragraph{Conjecture:} The conclusion of Theorem 4.3.13 remains valid even when some or all distinct nonzero eigenvalues are of geometric multiplicity exceeding 1.

The conjecture is not true as we can see in the following example. Let $B=I_3$ (with $P=I_3$) and $A=\begin{bmatrix}
           0 & 1  & 0 \\                                                                                                                   
            0  & 1  & 0\\
            0  & 0  & 0
            \end{bmatrix}$. It is easy to see that $A\stackrel{\sharp}\leq B$ but $A\neq I_3$ $\textrm{diag}\left(D_1, D_2, D_3 \right)I_3^{-1} $.

 However, the conjecture is not far from being true since the nonsingular matrix in the Jordan decomposition of $A$ is not necessarily equal to the nonsingular matrix $P$ in the Jordan decomposition of $B$ but the Jordan blocks of $A$ satisfy the conclusion of \cite[Theorem 4.3.13]{MiBhMa}. We can establish the following result which characterizes the matrices that are above a given matrix $A$ in terms of the Jordan decomposition of the matrix $A$ and answer the question that arises behind the conjecture.
 
\begin{lemma}\label{FormaJornalAbove} Let $A$ and $B$ be matrices in $ \mathbb{C}_1^n$ and $A= P \textrm{diag}\left(J_1, \dots, J_t, O \right)P^{-1}$ the Jordan decomposition of $A$, where $J_1, \dots, J_t$ are nonsingular Jordan blocks and $P$ is a nonsingular matrix. If $A\stackrel{\sharp}\leq B$ then $B= P\textrm{diag}\left(J_1, \dots, J_t, X \right)P^{-1}$ where $X$ is a matrix of adequate size.  
\end{lemma}
\begin{proof} Let $B= P\begin{bmatrix}
     X_{11} &   \cdots  & X_{1t} & X_{1(t+1)}\\                                                                                  
     \vdots &  \ddots &\vdots  &\vdots \\                                                                                                       
     X_{t1} &   \cdots & X_{tt} & X_{t(t+1)}\\
     X_{(t+1)1} & \cdots & X_{(t+1)t} & X_{(t+1)(t+1)}\\
      \end{bmatrix} 
P^{-1}$ for $X_{ij}$ being matrices of adequate size. If $A\stackrel{\sharp}\leq B$ then, in particular, $A^2=AB$. So, 
\[\begin{bmatrix} J_1^2        & \cdots  & O & O\\                                                                                                                   
                                                                                                                       
     \vdots &   \ddots &\vdots  &\vdots \\                                                                                                                   
     O &   \cdots &  J_t^2 &  O\\
     O &   \cdots & O & O \\
     \end{bmatrix}
    = \begin{bmatrix}
    J_1 X_{11}  & \cdots  & J_1 X_{1t} & J_1 X_{1(t+1)}\\                                                                                                                   
                                                                                                                       
       \vdots & \ddots &\vdots  &\vdots \\                                                                                                                   
    J_t X_{t1}  & \cdots & J_t X_{tt} &  J_r X_{t(t+1)}\\
       O & \cdots & O & O \\
      \end{bmatrix}. \]
Since $J_i$ are nonsingular, we have that $X_{ii}= J_i$, for $i\in \{1, \cdots, t\}$ and $X_{ij}=O$, for $i\neq j$, $j\in \{1, \cdots, t+1\}$ and $i\in \{1, \cdots, t\}$. That is, $B= P\begin{bmatrix}
     J_1        & \cdots  & O & O\\                                                                                                                   
                                                                                                            
       \vdots & \ddots &\vdots  &\vdots \\                                                                                                                   
     O &   \cdots &  J_t &  O\\     
     X_{(t+1)1} &   \cdots & X_{(t+1)t} & X_{(t+1)(t+1)}\\
      \end{bmatrix}
P^{-1}$. Now from $A^2= BA$, we obtain that 
\[\begin{bmatrix} J_1^2        & \cdots  & O & O\\                                                                                                                   
                                                                                                              
      \vdots & \ddots &\vdots  &\vdots \\                                                                                                                   
      O & \cdots &  J_t^2 &  O\\
      O & \cdots & O & O \\
     \end{bmatrix}
    =\begin{bmatrix}
     J_1^2        & \cdots  & O & O\\                                                                                                                   
                                                                                                                       
       \vdots & \ddots &\vdots  &\vdots \\                                                                                                                   
       O & \cdots &  J_t^2 &  O\\     
     X_{(t+1)1} J_1 &   \cdots & X_{(t+1)t}J_t & O\\
      \end{bmatrix}. \]
      Thus $X_{(t+1) j}=O$ for every $j\in \{1, \cdots, t\}$, from where we clearly get the result. 
\end{proof}

\begin{corollary}\label{bloques}Let $A$ and $B$ be matrices in $ \mathbb{C}_1^n$ and $A= P \textrm{diag}\left(J_1, \dots, J_t, O \right)P^{-1}$ the Jordan decomposition of $A$, where $J_1, \dots, J_t$ are nonsingular Jordan blocks and $P$ is a nonsingular matrix. If $A\stackrel{\sharp}\leq B$ then all the nonsingular Jordan blocks of $A$ are Jordan blocks of $B$. 
\end{corollary}

The converse is not true. For example, if $B=
\begin{bmatrix}
            1 & 1  & 0 \\                                                                                                                   
            0  & 1  & 0\\
             0  & 0  & 1
             \end{bmatrix}$ and $A= P\begin{bmatrix}
            1 & 1  & 0 \\                                                                                                                   
            0  & 1  & 0\\
             0  & 0  & 0
             \end{bmatrix}P^{-1}$, where $P= \begin{bmatrix}
            1 & 1  & 1 \\                                                                                                                   
            0  & 1  & 3\\
             0  & 0  & 1
             \end{bmatrix}$, then it is easy to see that $A$ is not a predecesor of $B$. 
						
Let us observe that the maximal elements of $\left( \mathbb{C}_1^n, \stackrel{_{\sharp}}\leq  \right)$ are nonsingular matrices. Indeed, let $B$ be decomposed as in (\ref{hartwigB}). If we consider the nonsingular matrix $C = U \begin{bmatrix}
   \Sigma K  &   (\Sigma  - K^{-1} ) L  \\
   O & I_{n-r} \\
	\end{bmatrix} U^*$, we have that $B\stackrel{_{\sharp}}\leq C$. 
 \begin{remark}\label{remarklongitudcadena2} The maximum length of any subchain in $\left( \mathbb{C}_1^n, \stackrel{_{\sharp}}\leq  \right)$ is $l+1$ where $1\leq l \leq n$. Indeed, let $B$ be a nonsingular matrix, $B= P \textrm{diag}\left(J_1, \dots, J_l \right)P^{-1}$ be the Jordan decomposition of $B$.  By Corollary \ref{bloques}, the sharp order does not ``split'' the Jordan blocks. So, the length of any chain of matrices that are below $B$ has at most $l+1$ elements. Moreover, if we consider the matrices $A_i=P \textrm{diag}\left(F_1, \dots, F_l \right)P^{-1}\in \mathbb{C}_1^{n}$ where $F_k=\begin{cases}
        J_k & \text{ if } k\leq i \\                                                                                                                                               
        O & \text{ otherwise}  
		\end{cases}
$, for each $i\in \{1,\dots,l\}$, then we obtain a chain 
\[O \stackrel{\sharp}\leq  A_1 \stackrel{\sharp}\leq \dots \stackrel{\sharp}\leq A_l=B,\] 
with $l+1$ elements of maximum length.
\end{remark}

In \cite{Mi} Mitra showed that $\left( \mathbb{C}_1^n, \stackrel{_{\sharp}}\leq  \right)$ with $n\geq 4$ is not a lower semilattice. As a consequence of Theorem \ref{teo:tj1o2}, we can extend this result to $n=3$. Moreover, we have the following theorem.

\begin{theorem} The poset $\left( \mathbb{C}_1^n, \stackrel{_{\sharp}}\leq  \right)$ is a lower semilattice if and only if $n\leq 2$. 
\end{theorem}
\begin{proof} If $n\geq 3$, then $\left( \mathbb{C}_1^n, \stackrel{_{\sharp}}\leq  \right)$ is not a lower semilattice. To see this it is enough to take a matrix $B$ where $\Sigma K$ has one eigenvalue such that its geometric multiplicity is greater than or equal to 3 and construct $T_3$ and $T_4$ as in the proof of the Theorem \ref{teo:tj1o2}. Then $T_3\wedge T_4$ does not exist in $[O, B]$ and also it is clear that $T_3\wedge T_4$ does not exist in $\mathbb{C}_1^n$. 

It is trivial that $\left( \mathbb{C}_1^1, \stackrel{_{\sharp}}\leq  \right)$ is a lower semilattice.

Let us prove that the poset $\left( \mathbb{C}_1^2, \stackrel{_{\sharp}}\leq  \right)$ is a lower semilattice. In order to do that, let $B_1,B_2\in \mathbb{C}_1^2$ where $B_1\neq B_2$. If the rank of any of them is 1 or 0 then $B_1\wedge B_2$ clearly exists in $\mathbb{C}_1^2$. Let us suppose that $B_1$ and $B_2$ have both rank equal to 2, that is, $B_1$ and $B_2$ are maximal in the set $\mathbb{C}_1^2$. Since $O$ is the least element in $\mathbb{C}_1^2$, the set of lower bounds of $\{B_1, B_2\}$ is nonempty. If $O$ is the only element of this set, then $B_1\wedge B_2$ exists and $B_1\wedge B_2=O$. We want to show that if there exists a matrix $A\neq O$ in the set of lower bounds then it is unique. In such a case, $A=B_1\wedge B_2$. 

We will argue by contradiction. Let $A_1, A_2$ be two distinct matrices of rank 1 which are both lower bounds of $\{B_1, B_2\}$. Let $\lambda_1$ and $\lambda_2$ be the unique nonzero eigenvalue of $A_1$ and $A_2$ respectively. If $\lambda_1\neq \lambda_2$ then $\lambda_1$ and $\lambda_2$ must be the eigenvalues of $B_1$ and $B_2$. Then $B_1= B_2$ by Lemma \ref{FormaJornalAbove}. Let us suppose now that $\lambda_1=\lambda_2=\lambda$. Then $\lambda $ must be an eigenvalue of $B_1$ and $B_2$. Neither $B_1$ nor $B_2$ can have a unique Jordan block. In fact, if $B_1= Q\begin{bmatrix}                                                                                                                                                                                                     \lambda & 1 \\
 0 & \lambda                                                                                                                                                                                                   \end{bmatrix}Q^{-1}$ and since $A_1\stackrel{\sharp}\leq B_1$, $A_1\neq B_1$, then by Theorem \ref{caracterizacionlibro} we obtain that $A_1=O$, which is a contradiction. 

Suppose now that $B_1= Q\begin{bmatrix}                                                                                                                                                                                                     \lambda & 0 \\
 0 & \mu                                                                                                                                                                                                   \end{bmatrix}Q^{-1}$ where $\lambda\neq \mu$. By Theorem \ref{caracterizacionlibro}, $A_1= Q\begin{bmatrix}                                                                                                                                                                                                     \lambda & 0 \\
 0 & 0                                                                                                                                                                                                   \end{bmatrix}Q^{-1}$ and $A_2= Q\begin{bmatrix}                                                                                                                                                                                                     0 & 0 \\
 0 & \mu                                                                                                                                                                                                   \end{bmatrix}Q^{-1}$, or vice versa. Now, by Lemma \ref{FormaJornalAbove}, it follows that $B_1$ is the unique upper bound of $\{A_1, A_2\}$ and then $B_1= B_2$, which is a contradiction. 
 
 So, $B_1= Q\begin{bmatrix}                                                                                                                                                                                                     \lambda & 0 \\
 0 & \lambda                                                                                                                                                                                                   \end{bmatrix}Q^{-1}= \lambda I_2= B_2$ and we have again a contradiction. \qedhere
 \end{proof}

\section{Solutions of some matrix equations}\label{sec:matrixeq}

In this section we give a characterization of solutions of matrix systems via the posets $\tau$ and $\delta$.

Let us recall that a matrix $B\in \mathbb{C}^{n\times n}$ is an EP-matrix if $\mathcal{R}(B)= \mathcal{R}(B^*)$. 
 
\begin{lemma}\label{lemma:sistemas} Let $B\in \mathbb{C}_1^{n}$ be an EP-matrix written as in (\ref{hartwigB}). Then, a matrix $S$ is a solution of the system  $\begin{cases}
                                         BX=XB \\
                                         X^2= X
                                        \end{cases}
$  if and only if $S=U\textrm{diag}(T, W)U^*$ where $T\in\tau $,  $T\in \Crr$, and $W\in \mathbb{C}^{(n-r)\times (n-r)}$ is a projector.\label{sistema1} 
\end{lemma}

\begin{proof} Since $B$ is an EP-matrix, then $L=O$ when $B$ is written as in (\ref{hartwigB}). Let us write $S=U \begin{bmatrix}
      T &  V \\
      Z & W \\
      \end{bmatrix}U^*$ where $T\in \Crr$. Let $S$ be a solution of the system. If $BS=SB$ then $\begin{bmatrix}
     \Sigma K T &  \Sigma K V \\
     O & O \\
     \end{bmatrix} = \begin{bmatrix}
     T \Sigma K  &  O \\
     Z \Sigma K & O \\
     \end{bmatrix} $, so $\Sigma K T= T \Sigma K$, $Z \Sigma K=O$ and  $\Sigma K V=O$; hence $Z=O$ and $V=O$. If $S^2=S$, then $\begin{bmatrix}
     T^2  &  O \\
     O & W^2 \\ \end{bmatrix}=\begin{bmatrix}
     T &  O \\
     O & W\\     
     \end{bmatrix} $. Thus, $T\in\tau $ and $W^2=W$. The converse is straightforward. 
\end{proof}

\begin{corollary} \label{cor:projectorescommutan}
Let $B\in \mathbb{C}_1^{n}$ be written as in (\ref{hartwigB}) and consider the Jordan canonical form $J$ of $\Sigma K$, such that $t_j=1$ in (\ref{Jordan Form 2}) for all $j\in \{1,\dots, s\}$.
\begin{enumerate}[(a)]
	\item\label{coridemcommutan} If $B$ is a nonsingular matrix, then there exist $2^s$ projectors that commute with $B$.
   \item If $B$ is an EP-matrix such that $\rk (B)=n-1$, then the system $\begin{cases}
                                         BX=XB \\
                                         X^2= X
                                        \end{cases}$
has $2^{s+1}$ solutions. 
\end{enumerate}
\end{corollary}
\begin{proof}By Corollary \ref{sharpBoolean}, $\tau$ has $2^s$ elements.
 \begin{enumerate}[(a)]
\item If $B$ is a nonsingular matrix, then it is an EP-matrix and by Lemma \ref{lemma:sistemas} the projectors that commute with $B$ are $S=UTU^*$, 
with $T\in\tau$. Then (a) holds. 

\item If $\rk (B)=n-1$, by Lemma \ref{lemma:sistemas}, we have that the projectors $S$ that commute with $B$ are $S=U \begin{bmatrix}
     T &  O \\
     O & a \\
     \end{bmatrix}U^*$, with $T\in\tau$ and $a\in \{0, 1\}$. So, the system has $2^{s+1}$ solutions. \qedhere
\end{enumerate}
\end{proof}

\begin{lemma} Let $B\in  \mathbb{C}_1^{n}$ be an EP-matrix written as in (\ref{hartwigB}) and $S=U \textrm{diag}(T, W)U^*$ where $T\in \Crr$,  $T\in\tau $ and $W\in \mathbb{C}^{(n-r)\times (n-r)}$ is a projector, then $S$ is a solution of $\begin{cases}
                                         XB^k=B^kX \\
                                         X^2= X
                                        \end{cases}$, for each $k$ positive integer. 
\end{lemma}
\begin{proof} It follows easily by using the definition of $\tau$ and taking into account that $T\Sigma K= \Sigma K T$ implies $T(\Sigma K)^k= (\Sigma K)^k T$, for each $k$.
\end{proof}

Next we give a result where $B$ is not necessarily an EP-matrix. Its proof is also straightforward as in Lemma \ref{lemma:sistemas} by using the definition of $\tau$. 

\begin{lemma} Let $B\in \mathbb{C}_1^{n}$ be a matrix written as in (\ref{hartwigB}) and $S=U \textrm{diag}(T, O)U^*$ where $T\in\tau $. Then $S $ is a solution of the system $\begin{cases}
                                         XBX=BX \\
                                         X^2= X
                                        \end{cases}
$.
\end{lemma}

\begin{lemma} Let $B\in \mathbb{C}^{n\times n}$ be such that $B=PJP^{-1}$, where $J$ is the Jordan canonical form  of $B$. Then $S$ 
is a solution of $\begin{cases}
                                         B X=XB \\
                                         X^2= X
                                        \end{cases}$  if and only if $S=PTP^{-1}$, with $T\in\delta_J$.

\end{lemma}
\begin{proof} It follows easily by using the definition of $\delta_J$ and taking into account that $B S= SB$ iff $PJP^{-1} P P^{-1} S = SP P^{-1} P J P^{-1}$
iff $J (P^{-1} S P)= (P^{-1} S P) J$, and $S^2=S$ iff $(P^{-1} S P)^2=P^{-1} S P$.
\end{proof}

\begin{corollary} 
Let $B\in \mathbb{C}^{n\times n}$ be a nonsingular matrix such that has only one eigenvalue $\lambda$.
\begin{enumerate}[(a)]
	\item If the geometric multiplicity of $\lambda$ is 1 and $S$ is a solution of $\begin{cases}
                                         B X=XB \\
                                         X^2= X
                                        \end{cases}$   then $S\in \{O, I_n\}$.
   	 
 \item  If the geometric multiplicity of $\lambda$ is 2, that is $B= P \textrm{diag}(J_1(\lambda),J_2(\lambda))P^{-1}$, where $J_1(\lambda)\in \mathbb{C}^{q\times q}$ and $S$ is a solution of $\begin{cases}
                                         B X=XB \\
                                         X^2= X
                                        \end{cases}$ then $S=O$ or the geometric multiplicity of the eigenvalue 1 of $S$ is $q$, $n-q$, or $n$.
																				
																				\end{enumerate}
\end{corollary}

\begin{corollary}
Let $B\in \mathbb{C}^{n\times n}$ be a nonsingular matrix with Jordan decomposition $B=PJP^{-1}$, where $J$ has only one Jordan block for each eigenvalue, that is $t_j=1$ in (\ref{Jordan Form 2}) for all $j\in \{1,\dots, s\}$. Then, $S$ is a projector that commute with $B$ if and only if $S=P\textrm{diag}(D_1, \cdots, D_s)P^{-1}$, where $D_i\in \{O, I_{r_i}\}$. Thus, there exist $2^s$ projectors that commute with $B$.
\end{corollary}

\begin{remark} As we have characterized matrices $A$ that are below $B$ under the sharp partial order, we also have characterized solutions of $X^2= XB=BX$ in the case that $B\in \mathbb{C}_1^{n}$ is nonsingular.
\end{remark}

\section*{Declaration of competing interest}

There is no competing interest to declare.
\section*{Data availability}
No data was used for the research described in the article.

\section*{Funding} 
The authors were partially supported by projects PGI 24/L128 and PGI 24/ZL22, Departamento de Matemática, Universidad Nacional del Sur (UNS), Argentina. The third author was partially supported by Ministerio de Ciencia, Innovaci\'on y Universidades of Spain (Grant REDES DE INVESTIGACI\'ON, MICINN-RED2022-134176-T), by Universidad Nacional de Río Cuarto (Grant PPI 083/2020), and by Universidad Nacional de La Pampa, Facultad de Ingeniería (Grant Resol. Nro. 135/19).

\end{document}